\documentclass[12pt]{article}
\usepackage{amsfonts,amssymb,amsbsy,amsmath,amsthm,mathrsfs,enumerate,verbatim}
 \usepackage[colorlinks=true]{hyperref}
 \hypersetup{urlcolor=blue, citecolor=red}
\newtheorem{theorem}{Theorem}[section]
\newtheorem{proposition}[theorem]{Proposition}
\newtheorem{lemma}[theorem]{Lemma}
\newtheorem{follow}[theorem]{Corollary}
\newtheorem{assumption}[theorem]{Assumption}
\newtheorem{pr}[theorem]{Proposition}
\theoremstyle{definition}

\newcommand{\bel}{\begin{equation} \label}
\newcommand{\ee}{\end{equation}}

\newcommand{\pd}{\partial}

\newcommand{\R}{{\mathbb R}}

\newcommand{\F}{{\mathcal F}}

\newcommand{\Q}{{\mathcal Q}}

\def\beq{\begin{equation}}
\def\eeq{\end{equation}}
\newcommand{\bea}{\begin{eqnarray}}
\newcommand{\eea}{\end{eqnarray}}
\newcommand{\beas}{\begin{eqnarray*}}
\newcommand{\eeas}{\end{eqnarray*}}
\newcommand{\Pre}[1]{\ensuremath{\mathrm{Re} \left( #1 \right)}}
\newcommand{\Pim}[1]{\ensuremath{\mathrm{Im} \left( #1 \right)}}
{

\begin{document}

\begin{center}
{\Large \bf Carleman estimate for infinite cylindrical quantum domains and application to inverse problems}

\medskip

%\today
\end{center}

\medskip

\begin{center}
{\sc  \footnote{CPT, CNRS UMR 7332, Universit\'e d'Aix-Marseille, 13288 Marseille, France \& Universit\'e du Sud-Toulon-Var, 83957 La Garde, France.}{Yavar Kian}, \footnote{Institute of Mathematics, Faculty of Mathematics and Computer Science, Jagiellonian University, 30-348 Krak\'ow, Poland.}{Quang Sang Phan},\footnote{CPT, CNRS UMR 7332, Universit\'e d'Aix-Marseille, 13288 Marseille, France \& Universit\'e du Sud-Toulon-Var, 83957 La Garde, France.}{Eric Soccorsi}}
\end{center}

\begin{abstract}
We consider the inverse problem of determining the time independent scalar potential $q$ of the dynamic Schr\"odinger equation in an infinite cylindrical domain $\Omega$, from one Neumann boundary observation of the solution. Assuming that $q$ is known outside some fixed compact subset of $\Omega$, we prove that $q$ may be Lipschitz stably retrieved by choosing the Dirichlet boundary condition of the system suitably. Since the proof is by means of a global Carleman estimate designed specifically for the Schr\"odinger operator acting in an unbounded cylindrical domain, the Neumann data is measured on an infinitely extended subboundary of the cylinder. 
\end{abstract}

\medskip



\tableofcontents

%%%%%%%%%%%%%%%%%%%%%%%%%%%%%%%%%%%%%%%%%%%%%%%%%%%
%%%%%%%%%%%%%%%%%%%%%%%%%%%%%%%%%%%%%%%%%%%%%%%%%%%

\section{Introduction}
\label{sec-intro} 
\setcounter{equation}{0}

\subsection{What we are aiming for}

In the present paper we consider the infinite cylindrical domain $\Omega=\omega \times \R$, where $\omega$ is a
connected bounded open subset of $\R^{n-1}$, $n\geq 2$, with $C^{2}$-boundary $\partial \omega$. Given $T>0$ we examine the following
initial boundary value problem 
\bel{eq1}
\left\{ \begin{array}{rcll} -i u'-\Delta u + q(x)u & = & 0, & {\rm in}\ Q:=(0,T) \times \Omega,\\ u(0,x) &= & u_0(x), & x\in \Omega,\\
u(t,x)& = & g(t,x),& (t,x)\in \Sigma:=(0,T) \times \Gamma, \end{array}\right. 
\ee 
where $\Gamma:=\pd \omega \times \R$ and the ' stands for
$\frac{\partial}{\partial t }$. Here $u_0$ (resp., $g$) is the initial (resp., boundary) condition associated to \eqref{eq1} and $q$ is a function of $x \in \Omega$ only. 

Since $\Gamma$ is unbounded we make the boundary condition in the last line of \eqref{eq1} more precise. Writing $x:=(x',x_n)$ with $x':=(x_1,\ldots,x_{n-1}) \in \omega$ for every $x \in \Omega$ we extend the mapping
\bea    
C_0^\infty ((0,T)\times \R; {\rm H}^2(\omega ))& \longrightarrow & {\rm L}^2((0,T) \times \R; {\rm H}^{3/2}(\partial \omega ))) \nonumber \\
v & \mapsto & [ (t,x_n) \in (0,T) \times \R \mapsto v(t,\cdot,x_n)_{|\partial \omega}], \label{g0} 
\eea
to a bounded operator from $\rm{L}^2 ((0,T)\times \R; {\rm H}^2(\omega ))$ into $\rm{L}^2((0,T)\times \R; {\rm H}^{3/2}(\partial \omega ))$, denoted by $\gamma_0$. Then
for every $u \in C^0([0,T];{\rm H}^2(\Omega))$ the above mentioned boundary condition reads $\gamma_0 u =g$.

The main purpose of this article is to extend \cite{BP}[Theorem 1], obtained for the dynamic Schr\"odinger operator in a bounded domain, to the case of the unbounded waveguide 
$\Omega$. That is to say, to prove Lipschitz stability in the determination of the scalar potential $q$ (assumed to be known outside some fixed compact subset of $\Omega$) from one boundary measurement of the normal derivative of the solution $u$ to \eqref{eq1}. 

The method of derivation of \cite{BP}[Theorem 1] is by means of a Carleman estimate for the Schr\"odinger equation, which was established in \cite{BP}[Proposition] for a bounded domain. Since the unknown part of $q$ is compactly supported, then it seems at first sight quite reasonnable to guess that this
question (arising from the problem of evaluating the electrostatic quantum disorder in nanotubes, see \S \ref{sec-pm} below) could well be answered by adapting the above technique to some suitable truncation (for some bounded domain of $\R^n$) of $u$. Nevertheless we shall prove that such a strategy necessarily adds unexpected ``control" terms (i.e. ``volume observations" of the solution) in the right hand side of the corresponding stability inequality, and is therefore inaccurate.

It turns out that this inconvenience can be avoided upon substituting some specifically designed Carleman estimate for the Schr\"odinger equation in a unbounded cylindrical domain to the one of \cite{BP}[Proposition 3]. This new global Carleman inequality, established in Proposition \ref{pr-unboundedCarl}, is the main novelty of this paper as it is the main tool for generalizing the Lipschitz stability inequality of \cite{BP}[Theorem 1] to the unbounded domain $\Omega$ under consideration.

It is worth noticing that the field of applications of Proposition \ref{pr-unboundedCarl} go beyond the study of inverse compactly supported quantum scalar coefficients PDE problems examined in this framework. As a matter of facts we prove with the aid of Proposition \ref{pr-unboundedCarl} in a companion article \cite{KiPhSo2}, that scalar potentials which are not necessarily known outside some given compact subset of $\Omega$, may nevertheless be H\"older stably retrieved from one boundary measurement of the normal derivative of $u'$.

Finally, let us mention that the Lipschitz stability estimate \cite{BP}[Theorem 1] was established under the additional technical assumption that both $u$ and $u'$ are time square integrable bounded functions of the space variables. As we aim to generalize this result to the case of the unbounded domain $\Omega$, we do the same here by supposing \eqref{as}. Nevertheless we stress out that sufficient conditions on $q$ and the initial state $u_0$ (which are not detailed in this text in order to avoid the inadequate expense of the size of the paper) ensuring assumption \eqref{as}, can be found in \cite{KiPhSo2}[Theorem 1.1].

\subsection{Physical motivation: estimating the electrostatic quantum disorder in nanotubes}
\label{sec-pm}
The equations of \eqref{eq1} describe the evolution of the wave function of a charged particle (in a ``natural" system of units where the various physical constants such as the mass and the electric charge of the particle are taken equal to one), under the influence of the ``electric" potential $q$. Moreover the quantum motion of this particle is constrained by the waveguide $\Omega$. 

Carbon nanotubes exhibit unusual physical properties, which are valuable for electronics, optics and other fields of materials science and technology. These peculiar nanostructures have a length-to-diameter ratio up to $10^8 /1$. This justifies why they may be rightfully modelled by infinite cylindrical domains such as $\Omega$. Unfortunately their physical properties are commonly affected by the inevitable presence of electrostatic quantum disorder, see e. g. \cite{CL, KBF}. This motivates for a closer look into the inverse problem of estimating the strength of the ``electric impurity potential" $q$ from the (partial) knowledge of the wave function $u$. 

\subsection{Existing papers}
There is a wide mathematical literature dealing with uniqueness and stability in inverse coefficient problems related to partial differential equations, see e. g. \cite{B, BY, CrS, I, IY,KY}.
For the stationnary (elliptic) Schr\"odinger equation, Bukhgeim and Uhlmann proved in \cite{BU} that the knowledge of the Dirichlet to Neumann (DN) map measured on some part of the boundary determines uniquely the potential. In dimension $n \geq 3$ this result was improved by Kenig, Sj\"ostrand and Uhlmann in \cite{KSU} and L. Tzou proved in \cite{T} that the electric potential (together with the magnetic field) depends stably on the Cauchy data even when the
boundary measurement is taken only on a subset that is slightly larger than half of the boundary.  Their method is essentially based on the construction of a rich set of ``optics geometric solutions" to the Dirichlet problem.
In these three papers, the knowledge of the DN map, taking infinitely many boundary measurements of the solution to the Schr\"odinger equation, is actually required.
 
The problem of stability in determining the
time-independent electric potential in the dynamic Schr\"odinger equation from a single boundary
measurement was treated by Baudouin and Puel in \cite{BP}. This result was improved by
Mercado, Osses and Rosier in \cite{MeOsRo}. In these two papers, the main
assumption is that the part of the boundary where the measurement is made satisfies
a geometric condition related to geometric optics condition insuring observability (see Bardos, Lebeau and Rauch \cite{BLR}).
This geometric condition was relaxed in \cite{BC} under the assumption that the potential
is known near the boundary.

In all the above mentioned articles the Schr\"odinger equation is defined in a bounded domain.
In the present paper we rather investigate the problem of determining the scalar potential of the Schr\"odinger equation in an infinite cylindrical domain. There are only a few mathematical papers dealing with inverse 
boundary value problems in an unbounded domain available in the mathematical literature. In \cite{LU} Li and Uhlmann prove uniqueness in the determination of the scalar potential in an infinite slab from partial DN map. In \cite{CCG} Cardoulis, Cristofol and Gaitan obtain Lipschitz stability from a single lateral measurement performed on one side of an unbounded strip. 

More specifically for an inverse boundary value problem stated in the waveguide geometry examined in this work we refer to \cite{CS, CKS} where stability is claimed for various coefficients of the Schr\"odinger equation from the knwoledge of the DN map. Here we investigate the same type of problems, but in absence of any information given by the DN map, by a method based on an appropriate Carleman estimate. We refer to \cite{A, BP, T} for actual examples of this type inequalities designed for the Schr\"odinger equation. The original idea of using a Carleman estimate to solve inverse problems goes back to the pioneering article \cite{BK} by Bukhgeim and Klibanov. Since then this technique has then been widely and succesfully used 
by numerous authors, see e.g. \cite{BP, B, BCS, BY, CS, Im, IY2, IsY, KY, MeOsRo}, in the study of inverse wave propagation, elasticity or parabolic problems.

\subsection{Main results}
In this section we state the main results of this article and briefly comment on them. 

For $M>0$ and $p \in W^{2,\infty}(\Omega)$ fixed, we define the set of ``admissible scalar potentials" as 
$$ \Q_M(p):= \{ q \in W^{2,\infty}(\Omega;\R),\ \|  q \|_{W^{2,\infty}(\Omega)}  \leq M\ \mbox{and}\ q(x)=p(x)\ \mbox{for a.e.}\ x \in \Gamma \}. $$
Here the identity $q_{|\Gamma} = p_{\Gamma}$ is understood in the sense of the usual trace operator from ${\rm L}^\infty(\R,{\rm H}^2(\omega))$ into
${\rm L}^\infty(\R,{\rm H}^{3 / 2}(\partial \omega))$. By selecting $q \in \Q_M(p)$ we thus enjoin fixed value to $q$ on the boundary, which is the mesurement on $\Gamma$ of the scalar potential we want to determine. Similar (or stronger) ``compatibility conditions" imposed on inverse problems coefficients have already been used in various contexts, see e.g. \cite{BC,BCS,CrS}.

We are concerned with the stability issue around any $q_1 \in \Q_M(p)$, i.e. we want to upper bound the ${\rm L}^2$ norm of $q_1-q_2$ by some increasing function of the difference $u_1-u_2$. That is to say that $q_2 \in \Q_M(p)$ and the solution $u_j$, for $j=1,2$, to \eqref{eq1}, where $q_j$ is substituted for $q$, are known, while $q_1$ is unknown.

\begin{theorem} 
\label{thmpr}
For $p \in W^{2,\infty}(\Omega)$, $M>0$, $\ell >0$ and $\alpha>0$ fixed, let $u_0 \in {\rm H}^4(\Omega;\R)$ obey
\bel{eqinitiale}
u_0(x) \geq \alpha >0,\ x \in \omega \times (-\ell,\ell),
\ee
let $q_j \in \Q_M(p)$, $j=1,2$, fulfill
\bel{eq1b}
q_1(x)=q_2(x),\ x \in \omega \times (\R \setminus (-\ell,\ell)),
\ee
and let $u_j$ denote the $C^1([0,T];{\rm H}^2(\Omega)  \cap {\rm H}_0^1(\Omega)) \cap C^2([0,T];{\rm L}^2(\Omega))$-solution to \eqref{eq1} associated to $u_0$, $g=\gamma_0 G$ and $q_j$, where
\bel{G}
G(t,x):= u_0(x)+ it (\Delta- p)u_0(x),\ (t,x) \in Q. 
\ee
Assume that
\bel{as}
u_2 \in {\rm H}^1(0,T;{\rm L}^{\infty}(\Omega)).
\ee
Then for every $L > \ell$, there exist $\Gamma_* \subset \partial \omega \times (-L,L)$ and a constant $C>0$ depending only on $L$, $T$, $M$, $\omega$ and $\Gamma_*$, such that we have
\bel{eq1a}
\| q_1-q_2 \|_{ {\rm  L}^2(\Omega)}
\leq  C  \left(  \|\partial _{\nu}    (u_1'-u_2') \|_ { {\rm L}^2((0,T ) \times \Gamma_*)} + \|
u_1-u_2  \| _{ {\rm H}^1( 0,T;{\rm H}^1(\omega \times \mathcal{S}_L ) )} \right),
\ee
with $\mathcal{S}_L:=(-L,-\ell) \cup (\ell, L)$.\\
Moreover there is a subboundary $\gamma_* \subset \partial \omega$, such that the estimate
\bel{eqa2}
\| q_1-q_2 \|_{ {\rm  L}^2(\Omega)}
\leq  C' \|\partial_{\nu} u_1'- \partial_{\nu}u_2' \|_ { {\rm L}^2((0,T) \times \gamma_* \times \R)},
\ee
holds for some positive constant $C'$ depending only on $\ell$, $T$, $M$, $\omega$ and $\gamma_*$.
\end{theorem}

Under the prescribed conditions \eqref{eqinitiale}--\eqref{G} and assumption \eqref{as}, the first statement \eqref{eq1a} of Theorem \ref{thmpr} claims Lipschitz stability in the determination of the scalar potential appearing in the dynamic Schr\"odinger equation in $\Omega$ from two different observations of the solution $u$ to \eqref{eq1}. 
\begin{enumerate}[$\bullet$]
\item The first one is a lateral measurement on some subboundary of $\partial \omega \times (-L,L)$ of the normal derivative 
$$ \partial_\nu u(t,x):= \nabla u(t,x) \cdot \nu(x),\ (t,x) \in \Sigma, $$
where $\nu$ is the outward unit normal to $\Gamma$. Since
\bel{normale}
\nu(x)=\nu(x')=\left( \begin{array}{c} \nu'(x') \\ 0 \end{array} \right),\ x=(x',x_n) \in \Gamma,
\ee
where $\nu'(x') \in \R^{n-1}$ denotes the outgoing normal vector to $\partial \omega$ computed at $x'$, we notice that
\bel{normale2}
\partial_\nu u(t,x)= \partial_{\nu'} u(t,x) := \nabla_{x'} u(t,x) \cdot \nu'(x'),\ (t,x)=(t,x',x_n) \in \Sigma,
\ee
where $\nabla_{x'}$ stands for the gradient operator w.r.t. $x' \in \omega$.
\item The second observation is an internal measurement of $u$ which is performed in each of the two ``slices" $\mathcal{S}_{L}^-:=\omega \times (-L,-\ell)$ and $\mathcal{S}_{L}^+:=\omega \times (\ell,L)$ of $\Omega$. Although the $\R^n$-Lebesgue measure of $\mathcal{S}_{L}^{\pm}$ can be made arbitrarily small by taking $L$ sufficiently close to $\ell$, this ``volume observation" cannot be removed from the right hand side of \eqref{eq1a} by taking $L$ asymptotically close to $\ell$ since the prefactor (i.e. the constant $C$) tends to infinity as $L$ goes to $\ell$.
\end{enumerate}
The occurence of this internal measurement of $u$ in \eqref{eq1a} is due to the unbounded geometry of $\Omega$. More precisely this is a direct consequence of the technique used for the derivation of the stability inequality \eqref{eq1a}, which is by means of a global Carleman estimate for the Schr\"odinger equation in a bounded domain. Indeed, this strategy requires a cut off function with first derivative supported in $(-L,-\ell) \cup (\ell,L)$, which gives rise to the measurement of $u$ in $\mathcal{S}_{L}^{\pm}$. Notice that the use of a Carleman estimate known to be valid in a bounded domain of $\R^n$ only, was made possible here since the difference $q_1-q_2$ is compactly supported in $\R^n$. A fact that follows from assumption \eqref{eq1b} expressing that the scalar potential to be retrieved is known outside some fixed bounded set.

One way to get rid of both volume observations simultanesously is to use a global Carleman estimate specifically designed for the unbounded quantum waveguide $\Omega$, which is stated in Proposition \ref{pr-unboundedCarl}. This yields \eqref{eqa2}, implying that the electrostatic quantum potential is now Lipschitz stably retrieved in $\Omega$ from only one lateral measurement of the normal derivative $\partial_\nu u$ on some subboundary of $\Sigma$. This result is similar to the one obtained in a bounded domain (under the same assumption as \eqref{as}) by Baudouin and Puel, see \cite{BP}[Theorem 1]. It should nevertheless be noticed that, contrarily to \eqref{eq1a}, and despite of the fact that the scalar potential under identification is assumed to be known outside a compact set, the Neumann data required in the right hand side of \eqref{eqa2} is measured on an infinitely extended subboundary of $\Sigma$.

Finally, it is worth mentioning that Theorem \ref{thmpr} applies for a wide class of subboundaries $\Gamma_*$. Indeed, given a $C^2$-domain $\mathcal{O}$ in $\R^n$ obeying
$$ \omega \times(-\ell,\ell) \subset \mathcal{O} \subset \omega \times (-L,L), $$
this is the case for any $\Gamma_* \supset \{x \in \partial \mathcal{O},\ (x-x_0)\cdot\nu_\mathcal{O}(x) \geq 0 \}$,
%$$ \Gamma_* \subset \{x \in \partial \mathcal{O},\ (x-x_0)\cdot\nu_\mathcal{O}(x) \geq 0 \} \cap \left( \partial \omega \times [-r,r] \right),\ r=\frac{\ell+L}{2}, $$
where $\nu_\mathcal{O}$ is the outward unit normal to $\partial \mathcal{O}$ and $x_0$ is arbitrary in $\R^n \setminus \overline{\mathcal{O}}$.
The same remark holds true for $\gamma_* \supset \{x'\in\partial \omega,\ (x'-x_0')\cdot\nu'(x') \geq 0 \}$, where
$x_0'$ is arbitrarily fixed in $\R^{n-1} \setminus \overline{\omega}$.

Theorem \ref{thmpr} immediately entails the followig uniqueness result.

\begin{follow}
Under the conditions of Theorem \ref{thmpr}, it holds true for all $q_j \in Q_M(p)$, $j=1,2$, that any of the two following assumptions
$$ 
\left\{
\begin{array}{cl} \partial_\nu u_1'(t,x) = \partial_\nu u_2'(t,x) & \mbox{for a.e.}\ (t,x) \in (0,T) \times \Gamma_* \\
u_1(t,x)=u_2(t,x) &  \mbox{for a.e.}\ (t,x) \in (0,T) \times \omega \times \mathcal{S}_L,
\end{array} 
\right.
$$
or
$$ \partial_\nu u_1'(t,x)= \partial_\nu u_2'(t,x)\ \mbox{for a.e.}\ (t,x) \in (0,T) \times \gamma_* \times \R,$$
yields $q_1=q_2$.
\end{follow}

\subsection{Outline}

The paper is organized as follows. Section \ref{sec-direct} deals with the direct problem associated to \eqref{eq1}. Namely \S \ref{sec-exiuni} gathers existence and uniqueness results for the solution to the dynamic Schr\"odinger equation in the infinite domain $\Omega$ and \S \ref{sec-linpro} is devoted to the study of the direct problem for the linearized system associated to \eqref{eq1}. In \S \ref{sec-timesym} the corresponding solution is suitably extended to a function of $[-T,T] \times \Omega$ which is continuous with respect to the time variable $t$. This is required by the method used in the proof of the stability inequalities \eqref{eq1a} and \eqref{eqa2}, which is the purpose of Section \ref{sec-si}. It is by means of the global Carleman estimate for the Schr\"odinger equation in a bounded domain (resp., in an unbounded waveguide) stated in Proposition \ref{pr-Carl} (resp., Proposition \ref{pr-unboundedCarl}) for \eqref{eq1a} (resp., for \eqref{eqa2}). Finally \S \ref{sec-proof} contains the completion of the proof of \eqref{eq1a}-\eqref{eqa2}.

\section{Analysis of the direct problem}
\label{sec-direct}
In this section we establish some existence, uniqueness and regularity properties of the solution to \eqref{eq1} needed for the analysis of the inverse problem
in section 3.
\subsection{Existence and uniqueness results}
\label{sec-exiuni}
This subsection gathers two existence and uniqueness results for \eqref{eq1}. We start by recalling from \cite{CKS}[Proposition 2.1] the following:
\begin{proposition}
\label{pr0} Let $M>0$. Then for all $q \in \mathcal{Q}_M$, $v_0\in {\rm H}_0^1(\Omega )\cap {\rm H}^2(\Omega )$ and $f\in
W^{1,1}(0,T;\rm{L}^2(\Omega ))$, there is a unique solution $v \in Z_0:=C([0,T]; {\rm H}_0^1(\Omega )\cap {\rm H}^2(\Omega
))\cap C^1([0,T]; {\rm L}^2(\Omega) )$ to the boundary value problem
$$
\left\{
\begin{array}{ll}
-i v' -\Delta v +q v = f, &{\rm in}\ Q,
\\
v(0,x )=v_0, & x \in \Omega ,
\\
v(t,x)=0, & (t,x) \in \Sigma.
\end{array}
\right.
$$
Moreover we have 
$$
\|v\|_{Z_0}\leq C\left(\|v_0\|_{{\rm H}^2(\Omega )}+\|f\|_{W^{1,1}(0,T; {\rm L}^2(\Omega ))}\right),
$$
for some constant $C>0$ depending only on $\omega$, $T$ and $M$.
\end{proposition}
Let the space
\bel{eqd9a} 
\mathscr{X}_0:=\gamma_0(W^{2,2}(0,T;{\rm H}^2(\Omega ))),
\ee
be equipped with the following quotient norm
$$
\|g\|_{\mathscr{X}_0}=\inf\{ \|G\|_{W^{2,2}(0,T;{\rm H}^2(\Omega ))};\ G\in W^{2,2}(0,T;{\rm H}^2(\Omega ))\; \mbox{satisfies}\ \gamma_0 G=g \}.
$$
Evidently, each $g\in \mathscr{X}_0$ admits an extension $G_0 \in W^{2,2}(0,T; {\rm H}^2(\Omega ))$
verifying
$$
\|G_0 \|_{W^{2,2}(0,T; {\rm H}^2(\Omega ))}\leq 2\|g\|_{\mathscr{X}_0}.
$$
Further, put 
\bel{eqd10} 
\mathscr{L}:=\{(u_0,g)\in {\rm H}^2(\Omega )\times \mathscr{X}_0;\; u_0=g(0,\cdot )\
\mbox{on}\ \Gamma \}. 
\ee
Then, setting $u=v+G_0$, where $v$ is defined by Proposition \ref{pr0} for
$f=i \partial_t G_0 + \Delta G_0 - q G_0$ and $v_0= u_0-G_0(0,.)$, we obtain the following existence and uniqueness result, which is similar to \cite{CKS}[Corollary 2.1].

\begin{theorem}
\label{thm0} 
Let $M >0$. Then for all $q \in \mathcal{Q}_M$ and all $(u_0,g) \in \mathscr{L}$ there is a unique solution $u\in
Z:=C([0,T]; {\rm H}^2(\Omega ))\cap C^1([0,T]; {\rm L}^2(\Omega) )$ to the boundary value problem \eqref{eq1}. Moreover we have the following estimate
\bel{eqd9} 
\|u\|_Z\leq C(\|u_0\|_{{\rm H}^2(\Omega)}+\|g\|_{\mathscr{X}_0}), 
\ee 
where $C$ is some positive constant depending only on $\omega$, $T$ and $M$.
\end{theorem}

\subsection{Linearized problem}
\label{sec-linpro}
In this subsection we establish some useful regularity properties of the solution to the linearized problem associated to \eqref{eq1} by applying the existence and uniqueness results stated in \S \ref{sec-exiuni}. 

To this purpose, we let $p$, $M$, $q_j$, $j=1,2$, $u_0$, $g$ and $G$ be the same as in Theorem \ref{thmpr}. Since $G \in W^{2,2}(0,T;{\rm H}^2(\Omega ))$ and $g(0,
\cdot)= u_0$ on $\Gamma$, we have $(u_0, g) \in \mathscr{L}$.
Applying Theorem \ref{thm0} there is thus a
unique solution $u_j \in Z$, $j=1,2$, to the following system
\bel{eq2}
\left\{ \begin{array}{rcll} -i u'_j-\Delta u_j + q_j(x)u_j & = & 0, & {\rm in}\ Q,\\ 
u_j(0,x) &= & u_0(x), & x\in \Omega,\\
u_j(t,x)& = & g(t,x),& (t,x)\in \Sigma. \end{array}\right. \ee
Further, differentiating \eqref{eq2} with respect to $t$ for $j=2$, we obtain that $u'_2$ is solution to the
boundary value problem
\bel{eq3}
\left\{ \begin{array}{rcll} 
-i u_2''-\Delta u'_2 + q_2(x)u'_2 & = & 0, & {\rm in}\ Q,\\ 
u'_2(0,x) &= & \tilde{u}_0(x), & x\in \Omega,\\
u'_2(t,x)& = & g'(t,x),& (t,x)\in \Sigma, 
\end{array}\right. 
\ee 
where $\tilde{u}_0:= i (\Delta -q_2)u_0$.
Since $G'(t,x)= i (\Delta-p)(x)$ for a.e. $(t,x) \in Q$ by \eqref{G} then we have $G' \in W^{2,2}(0,T;{\rm H}^2(\Omega ))$ and we get
$g'(0,x)=\tilde{u}_0(x)$ for a.e. $x \in \Gamma$ from the identity $q_2=p$ on $\Gamma$.
Consequently
$(\tilde{u}_0, g') \in \mathscr{L}$. Therefore $u_2' \in Z $ from \eqref{eq3} and Theorem \ref{thm0}. As a
consequence we have
$$u_2 \in C^1([0,T]; {\rm H}^2(\Omega ))\cap C^2([0,T]; {\rm L}^2(\Omega) ). $$

Moreover, it follows from \eqref{eq2} that $u:=u_1-u_2 $ is solution to the linearized system 
\bel{eq4}
\left\{ \begin{array}{rcll} -i u'-\Delta u + q_1u & = & f, & {\rm in}\ Q,\\ 
u(0,x) &= & 0, & x\in \Omega,\\
u(t,x)& = & 0,& (t,x)\in \Sigma, 
\end{array} \right. 
\ee with $f:= (q_2-q_1) u_2 $. Since $u_2 \in Z$ and $
q_1-q_2 \in  {\rm L}^{\infty}(\Omega)$ then $f \in W^{1,1}(0,T;\rm{L}^2(\Omega ))$ so we find that $u \in Z_0$ by
Proposition \ref{pr0}.

Last, we deduce from \eqref{eq4} that $v:= u'$ satisfies
\bel{eq5}
\left\{ \begin{array}{rcll} 
-i v' -\Delta v + q_1v & = & f', & {\rm in}\ Q,\\ 
v(0,x) &= & i (q_2-q_1)(x) u_0(x), & x\in \Omega. \\
v(t,x)& = & 0,& (t,x)\in \Sigma, \end{array} \right. 
\ee 
with $f'= (q_2-q_1)u'_2 \in  W^{1,1}(0,T;\rm{L}^2(\Omega ))$.
Since $q_1=q_2$ on $\Gamma$ we get that $i(q_2-q_1)u_0 \in {\rm H}_0^1(\Omega )\cap {\rm H}^2(\Omega )$. Therefore 
$v= u'_1-u'_2 \in Z_0$ by Proposition \ref{pr0}. Bearing in mind that $u'_2 \in Z$ we deduce from this that $u'_1 \in Z$.
Summing up, we have obtained that
$$u_j \in C^1([0,T]; {\rm H}^2(\Omega ))\cap C^2([0,T]; {\rm L}^2(\Omega) ),\ j=1,2. $$

\subsection{Time symmetrization}
\label{sec-timesym}
%It is required that a Carleman inequality for the solution $w$ to \eqref{eq6} is 
%established in $(-T,T) \times \mathcal{O} $. To this purpose 
This subsection is devoted to showing that the solution $v \in Z_0$ to \eqref{eq5} may be extended to a solution of the same system over the time span $[-T,T]$. 

To this end we extend $v $ (resp., $f'$) on $(-T,0) \times \Omega $ by setting
$v(t,x):=- \overline{v(-t,x)}$  (resp., $f'(t,x):=- \overline{f'(-t,x)}$) for $(t,x) \in (-T,0) \times \Omega$.
Since the initial conditions $v(0, \cdot) $ and $f'(0, \cdot)$  are purely complex valued according to \eqref{eq5}, 
the mappings $t \mapsto v(t,x)$ and $t \mapsto f'(t,x)$ are thus continuous at $t=0$ for
a.e. $x \in \Omega$. Therefore, we have $v \in C^0([-T,T]; {\rm H}_0^1(\Omega) \cap
{\rm H}^2(\Omega))\cap C^1([-T,T]; {\rm L}^2(\Omega) )$ and $f' \in W^{1,1}(-T,T;\rm{L}^2(\Omega ))$. Moreover $v$ is solution to \eqref{eq5} over the whole time span $(-T,T)$:
\bel{eq5b}
\left\{ \begin{array}{rcll} 
-i v' -\Delta v + q_1v & = & f', & {\rm in}\ Q_e:=(-T,T) \times \Omega,\\ 
v(0,x) &= & i (q_2-q_1)(x) u_0(x), & x \in \Omega. \\
v(t,x)& = & 0,& (t,x)\in \Sigma_e:= (-T,T) \times \Gamma. \end{array} \right. 
\ee 

\section{Carleman estimate for the dynamic Schr\"odinder operator}
\label{sec-carl}
The main purpose of this section is to derive a global Carleman estimate for the Schr\"odinger operator in an unbounded cylindrical domain $\mathcal{O} \times \R$, where $\mathcal{O}$ is a bounded domain of $\R^m$, $m \geq 1$, from the same type of results obtained for $\mathcal{O}$ by Baudouin and Puel in \cite{BP}[Proposition 3].

\subsection{What is known in a bounded domain} 
\label{sec-ce}
Given an arbitrary bounded domain $\mathcal{O} \supset \R^m$ with $C^2$ boundary, we consider the Schr\"odinger operator acting in 
$(C_0^{\infty})'((-T,T) \times \mathcal{O})$,
\bel{H} 
L_m := - {\rm i} \partial_t - \Delta,
\ee
where $\Delta$ denotes the Laplacian in $\mathcal{O}$.
Here we use the subscript $m$ to emphasize the fact that the operator $L_m$ acts in a subdomain of $(-T,T) \times \R^m$.

Next we introduce a function
$\tilde{\beta} \in {\rm C}^4(\overline{\mathcal{O}};\R_+)$ and an open subset $\Gamma_\mathcal{O}$ of $\partial \mathcal{O}$, satisfying the following conditions:
\begin{assumption}
\label{funct-beta}
\hfill \break \vspace*{-.5cm}
\begin{enumerate}[(i)]
\item $\exists C_0>0$ such that the estimate $|\nabla \tilde{\beta}(x)| \geq C_0$ holds for all $x \in \mathcal{O}$;
\item ${\partial}_{\nu} {\tilde{\beta}}(x) := \nabla {\tilde{\beta}}(x). \nu (x) < 0$ for all $x \in
    \partial \mathcal{O} \backslash \Gamma_\mathcal{O}$, where $\nu$ is the outward unit normal vector to $\partial \mathcal{O}$;
\item $\exists \Lambda_1>0$, $\exists \epsilon>0$  such that we have $\lambda |\nabla  \tilde{\beta}(x) \cdot
    \zeta|^2 + D^2 \tilde{\beta} (x,\zeta, \zeta) \geq \epsilon  |\zeta|^2$ for all $\zeta \in  \R^m$, $x \in \mathcal{O}$ and
    $\lambda  > \Lambda_1$, where $D^2 \tilde{\beta}(x):=\left( \frac{\partial^2 \tilde{\beta}(x)}{\partial x_i \partial x_j} \right)_{1 \leq i,j \leq m}$ and $D^2 \tilde{\beta} (x,\zeta, \zeta)$ denotes the $\R^m$-scalar product of $D^2 \tilde{\beta}(x) \zeta$ with $\zeta$.
\end{enumerate}
\end{assumption}
Notice that there are actual functions $\tilde{\beta}$ verifying Assumption \ref{funct-beta}, such as $\mathcal{O} \ni x \mapsto | x - x_0 |^2$, where $x_0$ is arbitrary in $\R^m \setminus \overline{\mathcal{O}}$ and $\Gamma_{\mathcal{O}} \supset\{ x \in \partial \mathcal{O},\ (x-x_0) \cdot \nu(x) \geq 0 \}$.

Further we put
\bel{defbeta} 
\beta:= \widetilde{\beta}+K,\ {\rm where}\ K:= r \|\tilde{\beta}\|_{\infty}\ {\rm for\
some}\ r>1, 
\ee 
and define the two following weight functions for $\lambda>0$: 
\bel{defphieta}
\varphi(t,x):=\frac{{\rm e}^{\lambda  \beta(x)}}{(T+t)(T-t)}\ {\rm and}\ \eta(t,x):=\frac{{\rm e}^{2\lambda K} -{\rm
e}^{\lambda \beta(x)}}{(T+t)(T-t)},\ (t,x) \in (-T,T) \times \mathcal{O}. 
\ee 
Finally, for all $s>0$ we introduce the two following
operators acting in $(C_0^{\infty})'((-T,T) \times \mathcal{O})$:
\bel{M1} 
M_{1,m} : = {\rm i} \partial_t +
\Delta  + s^2 |\nabla \eta |^2\ {\rm and}\ M_{2,m} : = {\rm i} s \eta' + 2 s \nabla \eta . \nabla  + s (\Delta \eta). 
\ee
It is easily seen that $M_{1,m}$ (resp., $M_{2,m}$) is the adjoint (resp., skew-adjoint) part of the operator ${\rm e}^{-s \eta} L_m {\rm e}^{s \eta}$,
where $L_m$ is given by \eqref{H}. 

Having said that we may now state the following global Carleman estimate, borrowed from \cite{BP}[Proposition 3], which is the main tool for the proof of Theorem \ref{thmpr}.

\begin{pr}
\label{pr-Carl} 
Let $\beta$ be given by \eqref{defbeta}, where $\tilde{\beta} \in {\rm C}^4(\overline{\mathcal{O}};\R_+)$
fulfills Assumption \ref{funct-beta}, let $\varphi$ and $\eta$ be as in \eqref{defphieta}, and let $L_m$, $M_{1,m}$ and
$M_{2,m}$ be defined by \eqref{H}-\eqref{M1}. Then we may find $s_0>0$ and constant
$C>0$, depending only on $T$, $\mathcal{O}$ and $\Gamma_\mathcal{O}$, such
that the estimate 
\bea
& & s  \| {\rm e}^{-s \eta}  \nabla w  \|_{{\rm L}^2((-T,T) \times \mathcal{O})}^2  +s^3  \| {\rm e}^{-s \eta} w \|_{{\rm L}^2((-T,T) \times \mathcal{O})}^2  + \sum_{j=1,2} \|  M_{j,m}  e^{-s \eta}w  \|_{{\rm L}^2(  (-T,T) \times \mathcal{O} )}^2 \nonumber  \\
%& & +  \sum_{j=1,2} \|  M_j ( e^{-s \eta}w ) \|_{{\rm L}^2(  (-T,T) \times \mathcal{O} )}^2 \nonumber \\
& \leq  & C  \left(  s
\| {\rm e}^{-s \eta} \varphi^{1/2}  ( \partial_{\nu} \beta)^{1/2}   \partial_{\nu} w  \|_{{\rm L}^2( (-T,T) \times \Gamma_\mathcal{O})}^2
+  \| {\rm e}^{-s \eta} L_m w \|_{{\rm L}^2((-T,T) \times \mathcal{O})}^2 \right), \label{Carl} 
\eea
holds for any real number $s \geq s_0$ and any function $w \in {\rm L}^2(-T,T;  {\rm H}^1_0( \mathcal{O}) )$
satisfying $L_m w \in {\rm L}^2((-T,T) \times \mathcal{O})$ and $\partial_{\nu} w \in {\rm L}^2(-T,T;{\rm
L}^2(\Gamma_\mathcal{O}))$.
\end{pr}
The proof of Proposition \ref{pr-Carl} can be found in \cite{BP}[section 2].

\subsection{Extension to the case of unbounded cylindrical domains}

Let $\omega$ be  $\Omega=\omega \times \R$ be the same as in \S \ref{sec-intro}. We use the same notation as in section \ref{sec-intro} and section \ref{sec-direct}, i.e. every point $x \in \Omega$ is written $x=(x',x_n)$ with $x'=(x_1,\ldots,x_{n-1}) \in \omega$, and we note $Q_e=(-T,T) \times \Omega$. We shall now derive from \ref{pr-Carl} a global Carleman estimate for the Schr\"odinger operator
\bel{Ln}
L_n :=-{\rm i} \partial_t - \Delta, 
\ee
acting in $(C_0^{\infty})'((-T,T) \times \Omega)$. Here $\Delta :=
\Delta_{x'} +  \partial ^2_{x_n}$ where $\Delta_{x'} := \Sigma_{j=1}^{n-1} \partial ^2_{x_j}$ is the Laplacian in $\omega$.
where $\nabla_{x'}$ (resp., $\Delta_{x'}$) stands for the gradient (resp., the Laplacian) operator w.r.t. $x' \in \omega$.

\begin{proposition}
\label{pr-unboundedCarl} 
For $n \geq 2$ fixed, let $\tilde{\beta}$ and $\gamma_*:=\Gamma_\mathcal{O}$ obey
Assumption \ref{funct-beta} with $\mathcal{O} = \omega$ and $m= n-1$, and let $\beta$ be given by \eqref{defbeta}. Extend $\beta$ to $\Omega=\omega \times \R$ by setting
$\beta(x)=\beta(x')$ for all $x=(x',x_n) \in \Omega$. Define $\varphi$ and $\eta$ (resp., the operators $M_{j,n}$ for $j=1,2$) by \eqref{defphieta} (resp., \eqref{M1}), upon substituting $\Omega$ for $\mathcal{O}$ and $n$ for $m$.
Then there are two constants $s_0>0$ and $C>0$, depending only on $T$, $\omega$ and $\gamma_*$,
such that the estimate
\bea 
& & s  \|  {\rm e}^{-s \eta}  \nabla_{x'} v   \|_{{\rm L}^2(Q_e)}^2
+s^3  \|
 {\rm e}^{-s \eta} v  \|_{{\rm L}^2(Q_e)}^2  + \sum_{j=1,2} \|   M_{j,n}  {\rm e}^{-s \eta} v  \|_{{\rm L}^2( Q_e )}^2
 \nonumber  \\
& \leq  & C  \left(  s \|   {\rm e}^{-s \eta} \varphi^{1/2}  ( \partial_{\nu} \beta)^{1/2}
\partial_{\nu}  v   \|_{{\rm L}^2( (-T,T) \times \gamma_* \times \R )}^2+  \|   {\rm e}^{-s \eta}   L_n v   \|_{{\rm L}^2(Q_e )}^2 \right),
\label{ec-inf}
\eea
holds for all $s \geq s_0$ and any function $v \in {\rm L}^2(-T,T;  {\rm H}^1_0( \Omega ) ) $ verifying $ L_n v \in
{\rm L}^2(Q_e)$ and $\partial_{\nu} v \in {\rm L}^2(-T,T;{\rm L}^2(\gamma_* \times \R ))$.
\end{proposition}
\begin{proof}
Since the function $\beta$ does not depend on the infinite variable $x_n$ then the same is true for $\varphi$ and $\eta$ according to \eqref{defphieta}. Therefore we have
\bel{M1n}
M_{1,n}  = {\rm i} \partial_t + \Delta
+ s^2 |\nabla_{x'} \eta |^2\ {\rm and}\ M_{2,n} = {\rm i} s \eta' + 2 s (\nabla_{x'} \eta) \cdot \nabla_{x'}  + s
(\Delta_{x'} \eta),
\ee
We denote by $\F_{x_n}$ the partial Fourier with respect to $x_n$, i.e. 
$$ (\F_{x_n} \psi)(x',k) = \widehat{\psi}(x',k) := \frac{1}{(2 \pi)^{1 \slash 2}} \int_{\R} e^{-i k x_n} \psi(x',x_n) d x_n,\ \psi \in {\rm L}^2(\Omega),\ x' \in \omega,\ k \in \R, $$
and consider the unitary group $U_t= e^{-i t \partial^2_{x_n}}$ in ${\rm L}^2(\Omega)$, defined by 
$$ U_t := \F_{x_n}^{-1}  e^{i t k^2} \F_{x_n},\ t \in \R, $$
where $\F_{x_n}^{-1}$ is the inverse transform to $\F_{x_n}$ and $e^{itk^2}$ stands for the multiplier by the corresponding complex number. Therefore we have
\bel{exp}   
(U_t \psi ) (x) = e^{-i t \partial^2_{x_n}} \psi (x) :=  \F^{-1}_{x_n}  \left( e^{itk^2} \widehat{\psi}(  \cdot, k ) \right) (x),\  \psi  \in {\rm L}^2
(\Omega),\ x \in \Omega,\ t \in \R.
\ee
We notice from \eqref{exp} that 
\bel{comm1}
[ i \partial_t , U_t ] = \partial_{x_n}^2 U_t = U_t \partial_{x_n}^2,\ t \in \R,
\ee
since $\partial^2_{x_n}$ commutes with $U_t$, and that
\bel{comm2}
[ \partial_{x_j}, U_t ] = [ e^{-s \eta}, U_t ] = 0,\ j=1,\ldots,n-1\ \mbox{and}\ s, t \in \R,\ 
\ee
since $\eta$ does not depend on $x_n$.
Let us now introduce the function
\bel{def-w}
w(t, x):= U_t v(t,x),\ (t,x) \in Q_e.
\ee
Then we have $w(\cdot, x_n) \in {\rm L}^2(-T,T;  {\rm H}^1_0(\omega) )$ and
$\partial_{\nu'} w(\cdot,x_n)  \in {\rm L}^2(-T,T;{\rm
L}^2(\gamma_*))$ for a.e. $x_n \in \R$.
% = \partial_{\nu} w(\cdot,x_n) from \eqref{normale}-\eqref{normale2}
Further, upon differentiating \eqref{def-w} w.r.t. $t$, we deduce from \eqref{comm1} that
\bel{def-dtw}
i \partial_t w %=-i \partial^2_{x_n}  e^{-i t \partial^2_{x_n}} v + e^{-i t \partial^2_{x_n}} \partial_t v 
= U_t(i \partial_t v + \partial^2_{x_n} v ).
\ee
This and the first commutator identity in \eqref{comm2} yields
$$
(- i\partial_t - \Delta_{x'}) w(t,x) = U_t L_n v(t, x),\ (t,x) \in Q_e,
$$
which entails
\bel{Lnminonew}
L_{n-1} w(\cdot,x_n)=U_t L_n v(\cdot,x_n),\ x_n \in \R.
\ee
As a consequence we have $L_{n-1} w(\cdot,x_n) \in {\rm L}^2((-T,T) \times \omega)$ for a.e. $x_n \in \R$, whence
\bea
& & s  \| {\rm e}^{-s \eta}  \nabla_{x'} w(\cdot,x_n)   \|_{{\rm L}^2((-T,T) \times \omega)}^2
+s^3  \| {\rm e}^{-s \eta}  w(\cdot,x_n)   \|_{{\rm L}^2((-T,T) \times \omega)}^2 \nonumber \\
& & + \sum_{j=1,2} \|  M_{j,n-1} e^{-s \eta} w(\cdot,x_n)   \|_{{\rm L}^2(  (-T,T) \times \omega  )}^2 \label{e1} \\
& \leq  & C  \left(  s \| {\rm e}^{-s \eta} \varphi^{1/2}  ( \partial_{\nu} \beta)^{1/2}   \partial_{\nu'}  w(\cdot,x_n)    \|_{{\rm
L}^2( (-T,T) \times \gamma_*)}^2 +  \| {\rm e}^{-s \eta} L_{n-1}   w(\cdot,x_n)  \|_{{\rm L}^2((-T,T) \times \omega)}^2
\right), \nonumber
\eea
for every $s \geq s_0$, by Proposition \ref{pr-Carl}.

The next step of the proof involves noticing from \eqref{comm2}--\eqref{def-dtw} that
\bel{co-w}
e^{-s \eta} w(\cdot,x_n)=U_t e^{-s \eta} v(\cdot,x_n),\ e^{-s \eta} \nabla_{x'} w(\cdot,x_n)=U_t e^{-s \eta} \nabla_{x'} v(\cdot,x_n),
\ee
and, with reference to \eqref{M1n}, that
\bel{co-M}
M_{j,n-1} e^{-s \eta} w(\cdot,x_n)=U_t M_{j,n} e^{-s \eta} v(\cdot,x_n),\ j=1,2,
\ee
for each $t \in (-T,T)$ and a.e. $x_n \in \R$. Similarly, we get 
\bel{co-L}
{\rm e}^{-s \eta} L_{n-1}   w(\cdot,x_n)= U_t {\rm e}^{-s \eta} L_n v(\cdot, x_n),\ t \in (-T,T),\ x_n \in \R,\
\ee
by combining the second identity of \eqref{comm2} with \eqref{Lnminonew}. Finally,
bearing in mind that $\varphi$ and $\beta$ are independendent of $x_n$, we deduce from \eqref{normale}-\eqref{normale2} and the second part of \eqref{comm2} that
\bel{co-trace}
{\rm e}^{-s \eta} \varphi^{1/2}  ( \partial_{\nu} \beta)^{1/2}   \partial_{\nu'} w(\cdot, x_n) = U_t {\rm e}^{-s \eta} \varphi^{1/2}  ( \partial_{\nu} \beta)^{1/2}   \partial_{\nu} v (\cdot, x_n),\ t \in (-T,T),\ x_n \in \R.
\ee
Therefore, putting \eqref{e1}--\eqref{co-trace} together, we obtain for $s \geq s_0$ and a.e. $x_n \in \R$ that
\bea
& & s  \|  U_t {\rm e}^{-s \eta}  \nabla _{x'} v(\cdot, x_n )  \|_{{\rm L}^2((-T,T) \times
\omega )}^2 +s^3  \| U_t {\rm e}^{-s \eta} v(\cdot, x_n )  \|_{{\rm L}^2((-T,T) \times \omega
)}^2
\nonumber  \\
& & + \sum_{j=1,2} \|  U_t M_{j,n}  e^{-s \eta} v(\cdot, x_n )  \|_{{\rm L}^2(  (-T,T)
\times \omega  )}^2
\nonumber \\
& \leq  & C  \left(  s \|  U_t {\rm e}^{-s \eta} \varphi^{1/2}  ( \partial_{\nu} \beta)^{1/2} 
\partial_{\nu}  v(\cdot, x_n )  \|_{{\rm L}^2( (-T,T) \times \gamma_*)}^2  \right.
+ \left.  \|  U_t {\rm e}^{-s \eta}  L_n v(\cdot, x_n) \|_{{\rm L}^2((-T,T)
\times\omega )}^2 \right).
\nonumber  
\eea
Integrating the above inequality w.r.t. $x_n$ over $\R$ and taking into account that the operator $U_t$ is isometric in ${\rm L}^2(\Omega)$ for each $t \in (-T, T)$
thus yields
\bea 
& & s  \| {\rm e}^{-s \eta}  \nabla _{x'} v  \|_{{\rm L}^2(Q_e)}^2 +s^3  \| {\rm e}^{-s \eta} v  \|_{{\rm L}^2(Q_e)}^2
+ \sum_{j=1,2} \| M_{j,n}  e^{-s \eta} v \|_{{\rm L}^2(Q_e)}^2
\nonumber   \\
& \leq  & C  \left(  s \|  U_t {\rm e}^{-s \eta} \varphi^{1/2}  ( \partial_{\nu} \beta)^{1/2} 
\partial_{\nu}  v \|_{{\rm L}^2( (-T,T) \times \gamma_* \times \R)}^2  \right.
+ \left.  \| {\rm e}^{-s \eta}  L_n v \|_{{\rm L}^2(Q_e)}^2 \right),
\label{e2} 
\eea
provided $s \geq s_0$. Last, setting $\Phi:={\rm e}^{-s \eta} \varphi^{1/2}  ( \partial_{\nu} \beta)^{1/2} 
\partial_{\nu}  v$ and using the unitarity of the operator $\F_{x_n}$ in ${\rm L}^2(\R)$, we derive from \eqref{exp} that
\beas
& & \|  U_t \Phi \|_{{\rm L}^2( (-T,T) \times \gamma_* \times \R)}^2 = \int_{-T}^T \int_{\gamma_*} \| \F_{x_n}^{-1} {\rm e}^{i t k} \widehat{\Phi}(t,x',\cdot) \|_{{\rm L}^2(\R)}^2 dx' dt\\
& = & \int_{-T}^T \int_{\gamma_*} \| {\rm e}^{i t k} \widehat{\Phi}(t,x',\cdot) \|_{{\rm L}^2(\R)}^2 dx' dt = \int_{-T}^T \int_{\gamma_*} \| \widehat{\Phi}(t,x',\cdot) \|_{{\rm L}^2(\R)}^2 dx' dt \\
& = & \| \Phi \|_{{\rm L}^2( (-T,T) \times \gamma_* \times \R)}^2,
\eeas
so \eqref{ec-inf} follows immediately from this and \eqref{e2}.
\end{proof}

\section{Proof of Theorem \ref{thmpr}}
\label{sec-si}
In view of section \ref{sec-carl} we are now in position to solve the inverse problem under study. 
More precisely we apply the Bugkhgeim-Klibanov method presented in \cite{BK} to prove the stability inequality \eqref{eq1a} (resp., the estimate \eqref{eqa2}) with the aid of Proposition \ref{pr-Carl} (resp., Proposition \ref{pr-unboundedCarl}).

\subsection{Proof of the stability inequality \eqref{eq1a}}
\label{sec-proof}
We first introduce the following notation used throughout the entire section: for all $d>0$ we write $\Omega_d:= \omega \times (-d,d)$.

Fix $L> \ell $, set $r:= (L+\ell )/2$ and consider a domain $\mathcal{O}$ in $\R^n$, with $C^2$ boundary $\partial \mathcal{O}$, obeying
\bel{domO}
\Omega_{\ell} \subset \Omega_r \subset \mathcal{O} \subset \Omega_L.
\ee
We next choose $\chi \in C_0^ \infty (\R_{x_n}, [0,1] )$ such that
$$ 
\chi(x_n) :=
\left\{ \begin{array}{ll} 1 & \textrm{if}\ |x_n| \leq \ell  \\ 0 & \textrm{if}\ | x_n| \geq r . \end{array} \right.
$$
Put $w:= \chi v$, where $v$ is the $C^0([-T,T]; {\rm H}_0^1(\Omega) \cap {\rm H}^2(\Omega)) \cap
C^1([-T,T]; {\rm L}^2(\Omega) )$-solution to \eqref{eq5b}; Then the restriction $w_{| \mathcal{O}}$ of $w$ to $\mathcal{O}$ satisfies
\bel{eq5c}
w_{|\mathcal{O}}  \in C^0([-T,T]; {\rm H}_0^1(\mathcal{O} ) \cap {\rm H}^2(\mathcal{O}))\cap
C^1([-T,T]; {\rm L}^2(\mathcal{O}) )
\ee
and is solution to the following system
\bel{eq6}
\left\{ \begin{array}{rcll} -i w'-\Delta w + q_1w & = & \chi f'- K v, & {\rm in}\ (-T,T) \times \mathcal{O},\\ 
w(0,x) &= & i \chi (q_2-q_1)(x) u_0(x), & x\in \mathcal{O}. \\
w(t,x)& = & 0,& (t,x)\in (-T,T) \times \partial \mathcal{O}, \end{array} \right. 
\ee 
where $ K:= [\Delta, \chi] =\ddot{\chi} + 2 \dot{\chi}  \partial_{ x_n} $. 
Here $\dot{\chi}$ (resp., $\ddot{\chi}$) is a shorthand for the first (resp., second) derivative of $\chi$.

As a preamble to the proof of the stability inequality \eqref{eq1a} we first establish two elementary technical results.

\subsubsection{Two auxiliary results}
\begin{lemma} 
\label{lem1}
For $s>0$, let $\phi :=  {\rm e}^{-s \eta} w$, where $w$ is defined by \eqref{eq5c}-\eqref{eq6}. Then we have 
$$   
J:= \|  {\rm e}^{-s \eta(0, \cdot )} w(0,\cdot ) \|_{{\rm
L}^2(\mathcal{O})}^2= 2  \Pim { \int_{(-T,0) \times  \mathcal{O}} M_1 \phi\overline{\phi}  dt dx},
$$
where $M_1$ stands for $M_{1,n}$.
\end{lemma}
\begin{proof}
In light of \eqref{defbeta}-\eqref{defphieta} it holds true that $\lim\limits_{\substack{t \downarrow (-T) }} \eta(t, x)= +\infty$ for all $x \in \mathcal{O}$, hence 
$$\lim\limits_{\substack{t
\downarrow (-T) }} \phi(t, x)= 0.$$ 
Therefore we have $J = \| \phi (0,\cdot) \|_{{\rm
L}^2(\mathcal{O})}^2 =  \int_{(-T,0) \times \mathcal{O}} \partial_t  | \phi |^2  dt dx$, from where we get that
\bel{J1} 
J = 2 \Pre{\int_{(-T,0) \times \mathcal{O}} \partial_t \phi \overline{\phi} dt dx} . 
\ee 
On the other hand, \eqref{M1} and the Green formula yield
\beas   
& & \Pim {\int_{(-T,0) \times \mathcal{O}} (M_1 \phi) \overline{\phi} dt dx} \\
%=  \Pim {\int_{(-T,0) \times \mathcal{O}}    ({\rm i} \partial_t \phi + \Delta \phi  + s^2 |\nabla \eta |^2 \phi ) \overline{\phi}  dt dx }  \\
& = & \Pre {\int_{(-T,0) \times \mathcal{O}}   \partial_t \phi  \overline{\phi} dt dx }  +   \Pim {
\int_{(-T,0) \times \mathcal{O}} \Delta \phi \overline{\phi}dt dx + s^2 \| \nabla \eta \phi  \|_{{\rm L}^2(-T,0) \times \mathcal{O}}^2}   \\
& = & \Pre{\int_{(-T,0) \times \mathcal{O}} \partial_t \phi \overline{\phi} dt dx } +  
\Pim{\| \nabla \phi \|_{{\rm L}^2((-T,0) \times \mathcal{O})}^2} = \Pre {\int_{(-T,0) \times \mathcal{O}} \partial_t \phi  \overline{\phi} dt dx },\\
%& = & \Pre {\int_{(-T,0) \times \mathcal{O}} \partial_t \phi  \overline{\phi} dt dx },  
\eeas 
so the result follows readily from this and \eqref{J1}.
\end{proof}

\begin{lemma}  
\label{lem2}
Let $w$ and $J$ be the same as in Lemma \ref{lem1}. Then we have
$$ J \leq  s^{-3/2}  I(w),\ s>0,$$
where
%\bea
%I(w) &:= & s  \| {\rm e}^{-s \eta}  \nabla w  \|_{{\rm L}^2((-T,T) \times \mathcal{O})}^2  +s^3  \| {\rm e}^{-s \eta} w \|_{{\rm L}^2((-T,T) \times \mathcal{O})}^2 \nonumber \\
%& &  + \sum_{j=1,2} \|  M_j  e^{-s \eta}w  \|_{{\rm L}^2(  (-T,T) \times \mathcal{O} )}^2. \label{I}
%\eea
\bel{I}
I(w) := s  \| {\rm e}^{-s \eta}  \nabla w  \|_{{\rm L}^2((-T,T) \times \mathcal{O})}^2  +s^3  \| {\rm e}^{-s \eta} w \|_{{\rm L}^2((-T,T) \times \mathcal{O})}^2 
+ \sum_{j=1,2} \|  M_j  e^{-s \eta}w  \|_{{\rm L}^2(  (-T,T) \times \mathcal{O} )}^2,
\ee
the notation $M_j$, for $j=1,2$, being a shorthand for $M_{j,n}$.
\end{lemma}
\begin{proof}
In view of Lemma \ref{lem1} and the Cauchy-Schwarz inequality, we have 
\beas
J 
%& = &  2\Pim { \int_{-T}^0  \int_{\mathcal{O}} (M_1 \phi) \overline{\phi} dt dx }
%2 s^{-3/2} \lambda^{-2} \Pim { \int_{-T}^0  \int_{\mathcal{O}} (M_1 \phi) ( s^{3/2} \lambda^{2} \overline{\phi}  ) \ dt dx  } \\
& \leq  & 2 \| M_1 \phi  \|_{{\rm L}^2((-T,T) \times \mathcal{O})} \|  \phi  \|_{{\rm L}^2( (-T,T) \times \mathcal{O} )} \\
% & \leq  & 2 s^{-3/2} \lambda^{-2} || M_1 \phi  ||_{{\rm L}^2(\mathcal{O}_T)} ||  s^{3/2} \lambda^{2} \phi  ||_{{\rm %L}^2(\mathcal{O}_T)} \\
 % \left(  \int_{-T}^0  \int_{\mathcal{O}}  | M_1 \phi | ^2 \right)^{1/2}
 % \left(  s^{3} \lambda^{4}  \int_{-T}^0  \int_{\mathcal{O}}  |  \phi | ^2 \right)^{1/2} \\
%   & \leq  &  s^{-3/2} \lambda^{-2}   \left(  \int_{-T}^0  \int_{\mathcal{O}}  | M_1 \phi | ^2 +
 % s^{3} \lambda^{4}  \int_{-T}^0  \int_{\mathcal{O}}  |  \phi | ^2 \right)  \\
& \leq &   s^{-3/2}  \left(   \|  M_1 ( e^{-s \eta}w ) \|_{{\rm L}^2( (-T,T) \times \mathcal{O})}^2
  +  s^3 \| {\rm e}^{-s \eta} w \|_{{\rm L}^2( (-T,T) \times \mathcal{O} )}^2   \right).
\eeas 
This and \eqref{I} yields the desired result.
\end{proof}

\subsubsection{Completion of the proof}
The next step of the proof involves majorizing $I(w)$ with the aid of the Carleman inequality of Proposition \ref{pr-Carl}. 
In view of \eqref{eq5c}-\eqref{eq6} and \eqref{I} the estimate
\bea 
I(w)  & \leq & C \left(    s 
\| {\rm e}^{-s \eta} \varphi^{1/2}  ( \partial_{\nu} \beta)^{1/2}  \partial_{\nu} w \|_{   {\rm L}^2( (-T,T) \times \tilde{\Gamma})}^2
 + \| {\rm e}^{-s \eta}(  \chi f'- K v)   \|_{   {\rm L}^2( (-T,T) \times \mathcal{O})}^2  \right) \nonumber \\
 & \leq  & C  \left(  s \| {\rm e}^{-s \eta} \varphi^{1/2}  ( \partial_{\nu} \beta)^{1/2} \partial_{\nu} w \|_{   {\rm L}^2( (-T,T) \times
 \tilde{\Gamma})}^2 \right.
  \nonumber \\
& & \hspace*{.5cm} + \left.  \sum_{j=0,1} \| {\rm e}^{-s \eta} \nabla^j v  \| _{   {\rm L}^2( (-T,T) \times  (\Omega_r \setminus
\Omega_{\ell}) )}^2+  \| {\rm e}^{-s \eta} \chi f'  \| _{{\rm L}^2( (-T,T) \times \Omega_r )}^2  \right), \label{E}
 \eea
holds with $\tilde{\Gamma}:=\Gamma_\mathcal{O}$ for $s \ge s_0$. Here and henceforth $C$ denotes some generic positive constant.
Bearing in mind that $\eta(0,x ) = \inf_{t \in (-T,T)} \eta(t,x)$ for all $x \in \Omega $ and that $f'= (q_2-q_1)u'_2$, we have
\bel{ee}
\| {\rm e}^{-s \eta} \chi f' \| _{{\rm L}^2( (-T,T) \times \Omega_r )} \leq C  \| {\rm e}^{-s \eta (0, \cdot)} (q_1-q_2)  \| _{{\rm L}^2( \Omega_{\ell} )}.
\ee
Here we used the fact that $\| u_2' \|_{{\rm L}^{\infty}((-T,T) \times \Omega_\ell)} < \infty$.
On the other hand, the function $\varphi \partial_{\nu} \beta$ being bounded in $(-T, T) \times \tilde{\Gamma}$ as well, we deduce from
\eqref{E}-\eqref{ee} that
\bea  
\label{eee}
I(w)  & \leq & C  \left(  s \| {\rm e}^{-s \eta(0,\cdot)} \partial_{\nu} w  \|_{{\rm L}^2( (-T,T) \times
 \tilde{\Gamma})}^2  + \sum_{j=0,1} \| {\rm e}^{-s \eta (0, \cdot)} \nabla^j v  \| _{{\rm L}^2( (-T,T) \times  (\Omega_r
\setminus \Omega_{\ell}) )}^2 \right. \nonumber \\
& & \hspace*{.7cm} \left. +\| {\rm e}^{-s \eta (0, \cdot)} (q_1-q_2)  \|_{{\rm L}^2( \Omega_{\ell} )}^2
\right) \nonumber \\
& \leq  & C  \left( \mathfrak{obs} + \| {\rm e}^{-s \eta (0, \cdot)} (q_1-q_2)  \|_{   {\rm L}^2( \Omega_{\ell} )}^2  \right),\ s \geq s_0,
\eea
where we have set
\bel{obs}
\mathfrak{obs} :=  s \| {\rm e}^{-s \eta(0,\cdot)} \partial_{\nu} w  \|_{{\rm L}^2( (-T,T)
\times \tilde{\Gamma})}^2  +  \sum_{j=0,1} \| {\rm e}^{-s \eta (0, \cdot)} \nabla^j v  \|_{{\rm L}^2( (-T,T) \times
(\Omega_r \setminus \Omega_{\ell} ))}^2. 
\ee
In light of Lemma \ref{lem2}, \eqref{eee}-\eqref{obs} then yields
\bel{d1}
J = \| {\rm e}^{-s \eta(0,\cdot )} \chi (q_1-q_2) u_0 \|_{ {\rm  L}^2(\mathcal{O} )}^ 2  \leq  C s^{-3/2}  \left(\mathfrak{obs} + \| {\rm e}^{-s \eta (0, \cdot)} (q_1-q_2)  \| _{   {\rm L}^2( \Omega_{\ell})}^2  \right), 
\ee
for all $s \geq s_0$. Further since $J \geq \alpha^2  \| {\rm e}^{-s \eta(0, \cdot )} (q_1-q_2) \|_{  {\rm  L}^2( \Omega_{\ell} ) }^ 2$ from the assumption \eqref{eqinitiale} we get that
\bel{d3}
(\alpha^2 -  C s^{-3/2})  \| {\rm e}^{-s \eta(0, \cdot )} (q_1-q_2) \|_{ {\rm  L}^2(\Omega_{\ell}  )}^ 2 \leq C s^{-3/2}  \mathfrak{obs},\ s \geq s_0,
\ee
from \eqref{d1}. Thus chosing $s \geq s_0$ so large that $\alpha^2 -  C
s^{-3/2} \geq  \alpha^2 \slash 2$ and taking into account that  $\inf_{x \in \Omega_{\ell}} {\rm e}^{-s \eta (0,x)} >0$, we derive from \eqref{obs}  and \eqref{d3} that
\bel{d4}
\| q_1-q_2 \|_{ {\rm  L}^2(\Omega_{\ell}  )}^ 2 \leq \ C  \left(  \| \partial_{\nu} w  \|_{{\rm L}^2(
(-T,T) \times  \tilde{\Gamma})}^2  +  \sum_{j=0,1} \| \nabla^j v  \| _{   {\rm L}^2( (-T,T) \times (\Omega_r
\setminus \Omega_{\ell} ) )}^2 \right).
\ee
Now, recalling from \S \ref{sec-timesym} that $\|\nabla^j v  \| _{{\rm L}^2( (-T,T) \times (\Omega_r \setminus \Omega_{\ell} ) )}=2 \| \nabla^j v  \| _{{\rm L}^2( (0,T) \times (\Omega_r
\setminus \Omega_{\ell} ) )}$ for $j=0,1$, and from the identity $w=\chi v$ that $\| \partial_{\nu} w  \|_{{\rm L}^2(
(-T,T) \times  \tilde{\Gamma})}=2\| \partial_{\nu} w  \|_{{\rm L}^2(
(0,T) \times  \tilde{\Gamma})}$, we derive from \eqref{d4} that
\bel{d5}
\| q_1-q_2 \|_{ {\rm  L}^2(\Omega )}^ 2  \leq C  \left(  \| \partial _{\nu} w \|_{ {\rm L}^2( (0,T) \times \tilde{\Gamma})}^2  +
\sum_{j=0,1} \|   \nabla^j  v  \| _{   {\rm L}^2( (0,T) \times (\Omega_r \setminus \Omega_{\ell} ) )}^2
\right).
\ee
Here we used the identity $ \| q_1-q_2 \|_{ {\rm  L}^2(\Omega_\ell)}=\| q_1-q_2 \|_{ {\rm  L}^2(\Omega )}$ arising from \eqref{eq1b}.

Last we set $\Gamma_*:=\tilde{\Gamma} \cap \left( \partial \omega \times [-r,r] \right)$ and $\Gamma_e := \tilde{\Gamma} \cap (\overline{\Omega_L} \setminus \overline{\Omega_r})$ in such a way that 
$$ \tilde{\Gamma} = \Gamma_* \cup \Gamma_e,\  \Gamma_* \cap \Gamma_e = \emptyset, $$
by \eqref{domO}. Since $\chi(x_n)=0$ for every $| x_n | \geq r$ then the function $w=\chi v$ is necessarily uniformly zero in $(0,T) \times (\Omega_L \setminus \Omega_r)$. This entails
$$ \partial_\nu w(t,\sigma) = 0,\ (t,\sigma) \in (0,T) \times \Gamma_e. $$
As a consequence we have 
$$\| \partial _{\nu} w \|_{ {\rm L}^2( (0,T) \times \tilde{\Gamma})} = \| \partial _{\nu} w \|_{ {\rm L}^2( (0,T) \times  \Gamma_*)} \leq C \| \partial _{\nu} v \|_{ {\rm L}^2( (0,T) \times  \Gamma_*)}. $$
Putting this together with \eqref{d5} and the identity $v= (u_1-u_2)'$, we end up getting \eqref{eq1a}.

\subsection{Proof of the stability estimate \eqref{eqa2}}
The proof of \eqref{eqa2} is very similar to the one of \eqref{eq1a}. Therefore, in order to avoid the inadequate expense of the size of this article, we shall skip some technical
details.

The function $v$ still denoting the solution to \eqref{eq5b}, we put 
$$ I(v):=  s  \|    {\rm e}^{-s \eta}  \nabla _{x'} v   \|_{{\rm L}^2(Q_e)}^2 +s^3  \|
{\rm e}^{-s \eta} v  \|_{{\rm L}^2(Q_e)}^2  + \sum_{j=1,2} \|   M_j e^{-s \eta} v  \|_{{\rm L}^2(Q_e)}^2,\ s>0,
$$
where we write $M_j$ for $j=1,2$, instead of $M_{j,n}$.
Then arguing in the exact same way as in the derivation of Lemmae \ref{lem1}-\ref{lem2}, we get that
\bel{eqa2b}
\|  {\rm e}^{-s \eta(0, \cdot )} v(0,\cdot ) \|_{{\rm
L}^2( \Omega)}^2  =\|  {\rm e}^{-s \eta(0, \cdot )} (q_1-q_2) u_0 \|_{{\rm
L}^2( \Omega)}^2  \leq  s^{-3/2}  I(v),
\ee
for all $s>0$.

On the other hand, applying Proposition \ref{pr-unboundedCarl} to \eqref{eq5b} we obtain
\bel{eqa3} 
I(v)   \leq  C \left(    s \| {\rm e}^{-s
\eta} \varphi^{1/2} ( \partial_{\nu} \beta)^{1/2}  \partial_{\nu} v \|_{   {\rm L}^2( (-T,T) \times \gamma_* \times
\R )}^2
 + \| {\rm e}^{-s \eta}  f'  \|_{   {\rm L}^2(Q_e )}^2  \right),
\ee
for every $s \geq s_0$. Further, taking into account that
$\eta(0,x' ) = \inf_{t \in (-T,T)} \eta(t,x')$ for all $x' \in \omega $ and $f'= (q_2-q_1)u'_2$, we derive from \eqref{as} that
\bel{eqa4} \|
{\rm e}^{-s \eta}  f' \| _{{\rm L}^2(Q_e )} \leq C'  \| {\rm e}^{-s \eta (0, \cdot)} (q_1-q_2)  \|
_{{\rm L}^2( \Omega )}, 
\ee
for some generic positive constant $C'$.
Since $\varphi \partial_{\nu} \beta $ is bounded in $(-T, T) \times
\gamma_* \times \R$, \eqref{eqa3}-\eqref{eqa4} then yield
\bel{eqa5}
I(v)  \leq C' \left( s \| {\rm e}^{-s \eta(0,\cdot)} \partial_{\nu} v \|_{{\rm L}^2(
(-T,T) \times \gamma_*  \times \R  )}^2 + \| {\rm e}^{-s \eta (0, \cdot)} (q_1-q_2)  \|_{   {\rm L}^2( \Omega )}^2 \right),\
s \geq s_0.
\ee 

Putting \eqref{eqa2b} and \eqref{eqa5} together we find that
\beas
& & \| {\rm e}^{-s
\eta(0,\cdot )} (q_1-q_2) u_0 \|_{ {\rm  L}^2(\Omega  )}^ 2 \\
& \leq  & C' s^{-3/2}  \left(s \| {\rm e}^{-s \eta(0,\cdot)} \partial_{\nu} v \|_{{\rm L}^2(
(-T,T) \times \gamma_*  \times \R  )}^2 + \| {\rm e}^{-s
\eta (0,\cdot)} (q_1-q_2)  \| _{   {\rm L}^2(\Omega)}^2 \right),\ s \geq s_0.
\eeas
From this and \eqref{eqinitiale} then follows for any $s \geq s_0$ that
\bel{eqa7}
(\alpha^2 -  C' s^{-3/2})  \| {\rm e}^{-s \eta(0, \cdot )} (q_1-q_2)
\|_{ {\rm  L}^2(\Omega)}^ 2 \leq C s^{-1/2} \| {\rm e}^{-s \eta(0,\cdot)} \partial_{\nu} v \|_{{\rm L}^2(
(-T,T) \times \gamma_*  \times \R  )}^2. 
\ee
Choosing $s \geq s_0$ so large that $\alpha^2 -  C' s^{-3/2} \geq  \alpha^2 \slash 2$ and using that  $\inf_{x \in
\omega  } {\rm e}^{-s \eta (0,x)} >0$, we deduce from \eqref{eqa7} that 
$$ 
\| q_1-q_2 \|_{{\rm  L}^2(\Omega  )}^ 2 \leq \ C'  \| \partial_{\nu} v  \|_{{\rm L}^2( (-T,T) \times \gamma_* \times \R
 )}^2. 
$$
Now \eqref{eqa2} follows readily from this upon recalling from \S \ref{sec-linpro} that $v= (u_1-u_2)'$ and from \S \ref{sec-timesym} that $\| \partial_{\nu} v  \|_{{\rm L}^2( (-T,T) \times \gamma_* \times \R  )}=2\| \partial_{\nu} v \|_{{\rm L}^2( (0,T) \times  \gamma_*  \times \R )}$.

\bigskip

\end{document}